\documentclass[11pt]{article}
\usepackage{graphics}
\usepackage{graphicx}
\usepackage[top=1in, left=1in, total={6.5in,9in}]{geometry}
\usepackage{hyperref}
\usepackage{enumitem}
\usepackage{tikz}
\usetikzlibrary{decorations}
\usetikzlibrary{decorations.pathmorphing}
\usetikzlibrary{decorations.pathreplacing}
\usetikzlibrary{decorations.markings}
\usetikzlibrary{decorations.text}
\usetikzlibrary{arrows.meta}
\usetikzlibrary{calc}
\usetikzlibrary[arrows,positioning, shapes, fit]
\usetikzlibrary{snakes}
\usepackage{amsmath,amssymb}
\usepackage{amsthm}
\newtheorem{thm}{Theorem}[section]
\newtheorem{lem}[thm]{Lemma}
\newtheorem{dfn}[thm]{Definition}
\newtheorem{prop}[thm]{Proposition}
\newtheorem{cor}[thm]{Corollary}
\newtheorem{obs}[thm]{Observation}
\newcommand{\Z}{\operatorname{Z}}
\newcommand{\F}{\operatorname{F}}
\newcommand{\specialt}{\operatorname{star}_{2,2,2}}
\usepackage{url}
\usepackage{comment}

\definecolor{lightgray1}{rgb}{0.85, 0.85, 0.85}
\definecolor{darkgray1}{rgb}{0.45, 0.45, 0.45}
\definecolor{grn3}{rgb}{0,0.35,0}
\begin{document}

\title{Well-failed graphs}
\author{Bonnie C. Jacob}
\maketitle

\abstract{
In this paper we begin the study of well-failed graphs, that is, graphs in which every maximal failed zero forcing set is a maximum failed zero forcing set, or equivalently, in which every minimal fort is a minimum fort.  We characterize trees that are well-failed.  Along the way, we prove that the set of vertices in a graph that are not in any minimal fort is identical to the set of vertices that are in no minimal zero forcing set, which allows us to characterize vertices in a tree that are in no minimal fort.  
}

\section{Introduction}

Over the past decade and a half, interest in the area of zero forcing has boomed.  Introduced because of its applications to combinatorial matrix theory \cite{aim2008zero} and control of quantum systems \cite{burgarth2009local, burgarth2013zero}, 
zero forcing picked up interest since its introduction and led to the introduction of several related concepts, including failed zero forcing, forts, and zero blocking, which are intimately connected to each other and are all addressed in this paper.  More information on the background of zero forcing can be found in \cite{hogben2022inverse}.  

A separate graph theoretic concept is the idea of well-$X$ graphs, where $X$ is a particular type of set.  A \emph{well-$X$} graph is a graph in which every minimal (or maximal, depending on $X$) $X$ set has the same cardinality.  For example, well-covered graphs are graphs in which every maximal independent set has the same cardinality \cite{hartnell1999well, plummer1993well}, and well-dominated graphs are graphs in which every minimal dominating set has the same cardinality \cite{anderson2021well, finbow1988well}.  

In \cite{grood2024well}, the authors applied the ``well'' idea to zero forcing.  In this paper,  we introduce well-failed graphs, which are graphs in which every maximal failed zero forcing set is maximum.  In doing so, we answer equivalent questions related to other parameters, most notably which graphs have the property that every minimal fort is minimum.  We characterize trees that are well-failed, and along the way, characterize vertices in a tree that are in no minimal fort, or equivalently, in every maximal failed zero forcing set. We establish that in any graph the set of vertices that are in every maximal failed zero forcing set is the same as the set of vertices that are in no minimal zero forcing set.  

\subsection{Definitions and notation}
Throughout this paper, we assume all graphs are simple (no loops or multiple edges), undirected, and finite.    Given a graph $G$, the vertex set is denoted $V(G)$ and the edge set $E(G)$.  If $uv \in E(G)$ for $u, v \in V(G)$, then we say that $u$ is a \emph{neighbor} of $v$, and vice versa. The \emph{open neighborhood} of $v$, denoted $N(v)$, is the set of neighbors of $v$, and the \emph{closed neighborhood} of $v$, denoted $N[v]$ is defined as $N[v]=N(v) \cup \{v\}$.  The \emph{degree} of a vertex $v$ is defined as $\deg(v)=|N(v)|$. For other standard graph theory terminology the reader may consult  \cite{CZ}.

A {\em pendent vertex} or \emph{pendant} is a vertex with degree one.  We refer to a vertex $v$ in a tree as a \emph{high-degree vertex} if $\deg(v) \geq 3$.  If $v, v'$ are pendants with common neighbor $w$, then we refer to $v$ and $v'$ as \emph{double pendants} and say that $w$ \emph{has a double pendant}.  

While there are many variations of zero forcing, this paper focuses on standard zero forcing, first formally introduced in \cite{aim2008zero}.  In standard zero forcing, hereafter referred to as simply ``zero forcing,'' a subset $S$ of the vertex set $V(G)$ is colored blue, and $V(G) \backslash S$ is colored white.  The following color change rule is then applied: if any blue vertex has exactly one white neighbor, then the white neighbor will change to blue.  The color change rule is applied repeatedly until no color changes are possible.  We call the set of blue vertices in the graph after no more color changes are possible the \emph{closure} of $S$, denoted $cl(S)$.  If $cl(S)=V(G)$, then  $S$ is a \emph{zero forcing set}.  The cardinality of a smallest zero forcing set is $\Z(G)$, the \emph{zero forcing number} of the graph.  

Failed zero forcing and stalled sets were first introduced  in  \cite{fetcie2015failed}.  A \emph{failed zero forcing set}, or for short, a \emph{failed set} is a set of vertices that is not a zero forcing set.  A \emph{stalled set} is a set of vertices $S$ such that $S \subsetneq V(G)$ and no color changes are possible from $S$.  The maximum cardinality among all failed sets is called the \emph{failed zero forcing number} of the graph, and is denoted $\F(G)$.  A \emph{maximum failed zero forcing set} is a failed zero forcing set that is of largest cardinality among all failed  sets of $G$.  That is, a maximum failed set $S$ is a failed set of $G$ with $|S|=\F(G)$.  Note that every stalled set is a failed set, but not the other way around.  Also, every maximum failed set is a stalled set.

\begin{dfn}
A \emph{maximal failed set} is a failed set $S$ such that adding any vertex from $V(G) \backslash S$ to $S$ produces a set that is not a failed set.  \end{dfn}

A maximum failed set is a maximal failed set but not the other  way around.  Also, every maximal failed set is a stalled set.

A \emph{fort}, first formally introduced in  \cite{fast2017novel, fast2018effects}, is a nonempty set of vertices $W$ such that $|N(v) \cap W| \neq 1$ for any vertex $v$ with $v \in V(G) \backslash W$.  That is, a fort is the complement of a stalled set, though forts and stalled sets were introduced independently.  For forts and stalled sets, we use minimality and maximality differently than we do for failed sets.  

\begin{dfn}
A \emph{maximal stalled set} is a stalled set that is   not a proper subset of a stalled set.  Similarly, a \emph{minimal fort} is a fort that does not contain another fort as a proper subset.  \end{dfn}

That is, for forts and stalled sets we use minimality and maximality with respect to inclusion.

A \emph{zero blocking set}, which is  the complement of a failed set, was defined \cite{beaudouin2020zero}.  We stay with the terminology of failed zero forcing and forts in this paper, but include mention of zero blocking sets for completeness.  A minimal zero blocking set is the complement of a maximal failed set: it is a zero blocking set such that if we remove any vertex from it, it is no longer a zero blocking set.  

\subsection{Preliminaries and basic results}
In this paper, we are interested in determining which graphs have the special property that every maximal failed set is a maximum failed set.  We could also ask on which graphs every minimal fort is also a minimum fort, or a similar question for stalled and zero blocking sets, but it turns out that these questions are actually all equivalent as shown in the following lemma.  

\begin{lem} \label{lem:equivalent}
The following are equivalent in a graph $G$ for $S \subseteq V(G)$.
\begin{enumerate}[label={(\arabic*)}]
\item  $S$ is a maximal stalled set \label{item:stalled}
\item $S$ is a maximal failed set  \label{item:failed}
\item $V \backslash S$ is a minimal fort. \label{item:fort}
\item $V \backslash S$ is a minimal zero blocking set.  \label{item:zeroblocking}
\end{enumerate}  
\end{lem}

\begin{proof}
The equivalence of  \ref{item:stalled} and \ref{item:fort} is clear from the definitions, as is the equivalence of  \ref{item:failed} and \ref{item:zeroblocking}.

To show that \ref{item:stalled} implies \ref{item:failed}, suppose $S$ is a maximal stalled set.  Then any superset of $S$ other than $V(G)$ allows a color change including $S \cup \{v\}$ for any $v \in V(G) \backslash S$. Then $S \subsetneq cl(S \cup \{v\})$, so $cl(S \cup \{v\}) = V(G)$ since $S$ is a maximal stalled set, and $S$ is therefore a maximal failed set.  

To show that \ref{item:failed}   implies  \ref{item:stalled},  suppose that $S$ is a maximal failed set.  Adding any vertex to $S$ results in a zero forcing set.  Then any superset of $S$ is not stalled.  That is, $S$ is a maximal stalled set, completing the proof.    
\end{proof}

Thus, in determining which graphs have the property that every maximal failed set is a maximum failed set, we answer the analogous question for forts, zero blocking sets, and stalled sets.  We call a graph with this property a well-failed graph.  

\begin{dfn}
A \emph{well-failed graph} is a graph in which every maximal failed set is a maximum failed set.
\end{dfn} 

We can immediately note the following for graphs that contain a double pendant.  

\begin{obs}
Suppose $v, v'$ are double pendants in a graph $G$, and for some maximal failed set $S$, $v, v' \notin S$.  Then $S=V(G) \backslash \{v, v'\}$.   Equivalently, if $v, v' \in W$ for some minimal fort $W$, then $W=\{v, v'\}$.  \label{obs:doublependants}
\end{obs}

Because blue-white colorings of paths come up multiple times in this paper, we introduce the following term.  On a path $P_n$, label the vertices $v_1, v_2, \ldots v_n$ starting with either end vertex.  Color blue every vertex of the form $v_{2j}$ for any $j$ such that $2j < n$.  We call this the \emph{standard coloring} of a path.  

The following fact was proved in the language of failed zero forcing in \cite{fetcie2015failed}.
\begin{prop}
\cite{fetcie2015failed} The white vertices of the standard coloring of the path form a minimum fort.  
\end{prop}

\begin{lem}
The path $P_n$ is well-failed if and only if $n \in \{1, 2, 3, 4, 6\}$.  \label{lem:path}
\end{lem}

\begin{proof}
For $n \in \{1,2\}$,  the only fort is $V(P_n)$.  For $n \in \{3,4\}$, the only forts consist of all vertices but one internal vertex.  For $n=6$, $\F(T)=2$.  For any internal vertex $v$, $\{v\}$ is a failed set, but is not maximal since there is at least one internal vertex at distance two from $v$ that we could add to produce a larger failed set.  Thus if  $n \in \{1, 2, 3, 4, 6\}$, $P_n$ is well-failed.

For $n=5$ or $n \geq 7$, let $S$ be the blue vertices of the standard coloring. The set $S'=\{v_3\} \cup S \backslash \{v_2, v_4\}$ is a maximal failed set with $|S'|=|S|-1$,  completing the proof. 
\end{proof}

While this paper focuses on well-failed graphs, we also find some relationships with well-forced graphs, which were introduced in \cite{grood2024well}.  A graph is \emph{well-forced} if every minimal zero forcing set is a minimum zero forcing set.  In \cite{grood2024well}, the authors also explored the related idea of \emph{zero-forcing-irrelevant} vertices, which were introduced in  \cite{bong2023isomorphisms} and are vertices that are in no minimal zero forcing set of the graph.  In \cite{bong2023isomorphisms}, the concept of irrelevant vertices was defined more generally in terms of $X$-sets, where an $X$-set satisfies a particular collection of axioms.  In this paper, we will discuss zero-forcing-irrelevant vertices, as well as \emph{fort-irrelevant vertices}, which are vertices that are in no minimal fort in the graph, and \emph{failed-zero-forcing-irrelevant vertices}, which are vertices that are in no maximal failed set.  However, it's important to note that forts and failed zero forcing sets do not satisfy the axioms in the definitions of $X$-sets, and therefore do not fall under the $X$-irrelevant umbrella from  \cite{bong2023isomorphisms}.

Star centers, called $B$-vertices when they were introduced in \cite{grood2024well}, are critical in characterizing both  well-forced trees and zero-forcing-irrelevant vertices in trees.  For any vertex $v$, let $L(v)$ denote the neighbors of $v$ with degree one, and let $L[v]=\{v\} \cup L(v)$.  If $G$ is a graph with a vertex $v$ such that $|L(v)| \geq 2$, then we can produce a new graph  $G' = G \backslash L[v]$.  We call the process of producing $G'$ from $G$ a \emph{star removal}.  We now define a star center.  
\begin{dfn}
In a graph $G$, define $B_0$ to be the set of vertices that have a double pendant.  Perform star removals on $G$ for all vertices in $B_0$ to create the graph $G_1$. Then $B_1$ is the set of vertices in $G_1$ that each have a double pendant.   Continuing, for $i \geq 1$, we construct $G_{i+1}$ by performing star removals on $G_i$ for all vertices in $B_i$.  The vertices in $B_0, B_1, B_2,$ etc. are said to be \emph{star centers} of $G$.  
\end{dfn}

An example of a tree with its star centers shaded, and the resulting forest after performing star removals for all star centers in $B_0$, is shown in Figure \ref{fig:bvertices}.

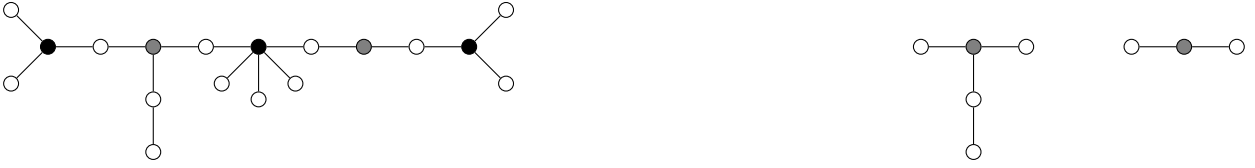
\begin{figure}
\begin{center}
\begin{tikzpicture}[auto, scale=0.7]
\tikzstyle{vertex}=[draw, circle, inner sep=0.7mm]
\node (v1) at (-0.7,0.7) [vertex]{};
\node (v2) at (-0.7,-0.7) [vertex]{};
\node (v3) at (0,0) [vertex, fill=black]{};
\node (v4) at (1,0) [vertex]{};
\node (v5) at (2,0) [vertex, fill=gray]{};
\node (v6) at (2,-1) [vertex]{};
\node (v7) at (2,-2) [vertex]{};
\node (v8) at (3, 0) [vertex]{};
\node (v9) at (4, 0) [vertex, fill=black]{};
\node (v10) at (3.3, -0.7) [vertex]{};
\node (v10b) at (4, -1) [vertex]{};
\node (v11) at (4.7, -0.7) [vertex]{};
\node (v12) at (5, 0) [vertex]{};
\node (v13) at (6, 0) [vertex, fill=gray]{};
\node (v14) at (7, 0) [vertex]{};
\node (v15) at (8, 0) [vertex, fill=black]{};
\node (v16) at (8.7, 0.7) [vertex]{};
\node (v17) at (8.7, -0.7) [vertex]{};
\draw (v1) -- (v3) -- (v2);
\draw (v3)--(v4)--(v5)--(v8)--(v9)--(v12)--(v13)--(v14)--(v15);
\draw(v5)--(v6)--(v7);
\draw(v10)--(v9)--(v10b);
\draw(v9)--(v11);
\draw(v16)--(v15)--(v17);
\end{tikzpicture}
\hfill
\begin{tikzpicture}[auto, scale=0.7]
\tikzstyle{vertex}=[draw, circle, inner sep=0.7mm]
\node (v4) at (1,0) [vertex]{};
\node (v5) at (2,0) [vertex, fill=gray]{};
\node (v6) at (2,-1) [vertex]{};
\node (v7) at (2,-2) [vertex]{};
\node (v8) at (3, 0) [vertex]{};
\node (v12) at (5, 0) [vertex]{};
\node (v13) at (6, 0) [vertex, fill=gray]{};
\node (v14) at (7, 0) [vertex]{};
\draw (v4)--(v5)--(v8);
\draw (v12)--(v13)--(v14);
\draw(v5)--(v6)--(v7);
\end{tikzpicture}
\caption{A tree $T$ with  $B_0$ and $B_1$ star centers shaded \textcolor{black}{black} and \textcolor{gray}{gray} respectively, and the forest $T_1$ after  all $B_0$ star removals have been performed on $T$.}
\label{fig:bvertices}
\end{center}
\end{figure}

\section{Irrelevant vertices}
In this section, we discuss fort-irrelevant vertices, that is, vertices that are not in any minimal fort.  Note that fort-irrelevant vertices are also the vertices that must be in every maximal failed set by Lemma \ref{lem:equivalent}. The main result of this section is Theorem \ref{thm:antifortimpliesirrelevant}, which shows that the set of fort-irrelevant vertices is the same as the set of zero-forcing-irrelevant vertices.  We also characterize failed-zero-forcing-irrelevant vertices in Corollary \ref{cor:fzfirrelevant}.  

The following fact was shown for zero forcing sets and forts in general \cite{brimkov2019computational} but we restate it applied to minimal zero forcing sets and minimal forts.

\begin{lem} \cite{brimkov2019computational} 
If $S$ is a minimal zero forcing set in $G$ and $W$ a minimal fort, then $S \cap W$ is nonempty. \label{lem:minimalzfsintersectminimalfort}
\end{lem}
Suppose $\mathcal{B}(G)$ is a collection of forts in $G$.  A \emph{cover} of $\mathcal{B}(G)$ is a set $C \subseteq V(G)$ such that $C \cap W$ is nonempty for each fort $W \in \mathcal{B}(G)$.  The cover $C$ is a \emph{minimal cover} of $\mathcal{B}(G)$ if no subset of $C$ is also a cover of $\mathcal{B}(G)$.
The following proposition appeared in \cite[Proposition 2.8]{brimkov2022minimal} for minimal covers of \underline{all} forts (as opposed to minimal covers of \underline{minimal} forts).  

\begin{prop}
A set $S \in V(G)$ is a  minimal zero forcing set of $G$ if and only if it is also a minimal cover of minimal forts of $G$.  \label{prop:minimalzfsisminimalcover}
\end{prop}  

\begin{proof}
Suppose $S$ is a minimal zero forcing set of $G$.  Then $S$ is a cover for the set of minimal forts of $G$  \cite{brimkov2022minimal}.  For any $S' \subsetneq S$, since $S$ is a minimal zero forcing set, $S'$ is not a zero forcing set.   Then $cl(S') \subsetneq V(G)$, and $cl(S')$ is a stalled set, giving us that $V(G) \backslash cl(S')$ is a fort.  Let $W \subseteq V(G) \backslash cl(S')$ be a minimal fort.  Then  $W  \cap S' = \emptyset$, and the set  $S'$ is not a cover of minimal forts. Hence, $S$ is a minimal cover of the set of minimal forts of $G$.  

For the reverse direction, suppose $S$ is a minimal cover of minimal forts of $G$.   Then $S$ is a zero forcing set  \cite{brimkov2022minimal}.  Suppose $S' \subsetneq S$ for some zero forcing set $S'$.  Since $S$ is a minimal cover, $S'$ is not a cover of minimal forts, and there exists a minimal fort $W$ such that $W \cap S'=\emptyset$, contradicting Lemma \ref{lem:minimalzfsintersectminimalfort}.  Hence, $S$ is a minimal zero forcing set, completing the proof. 
\end{proof}

We also establish that each vertex in a minimal fort must appear in some minimal zero forcing set. 

\begin{lem}
Suppose $v \in W$ for some minimal fort $W$ of $G$.  Then $v \in S$ for some minimal zero forcing set $S$.  \label{lem:inminimalfortinzfs}
\end{lem}
\begin{proof}
If $v$ is in every minimal fort of $G$, then $\{v\}$ is a minimal cover of minimal forts of $G$, and hence by Proposition \ref{prop:minimalzfsisminimalcover}, $\{v\}$ is a minimal zero forcing set.

Otherwise, let $W_1, W_2, \ldots, W_s$ be the minimal forts of $G$.  Assume that $v \in W_1, W_2, \ldots, W_k$ for some $k$ with $1 \leq k < s$, and $v \notin W_{k+1}, W_{k+2}, \ldots W_s$.  Since $W_i$ is a minimal fort for each $i$, $1 \leq i \leq s$, we must have that $W_j \not \subseteq W_i$ for any $j \neq i$.  Hence, $W_j \backslash W_1$ is nonempty for each $j \geq 2$.  For each $i>k$, pick $w_i \in W_i \backslash W_1$. 
 Then $\{w_i : i > k\}$ forms a cover of $\{W_{k+1}, W_{k+2}, \ldots W_s\}$.  If it is not a minimal cover, let $X \subseteq \{w_i : i > k\}$ be a minimal cover of $W_{k+1}, W_{k+2}, \ldots W_s$.  Then $\{v\} \cup X$ is a minimal cover of $W_1, \ldots, W_s$, and by Proposition \ref{prop:minimalzfsisminimalcover} a minimal zero forcing set of $G$, completing the proof.  
\end{proof}

\begin{thm}
A vertex $v$ is fort-irrelevant if and only if it is zero-forcing-irrelevant.    \label{thm:antifortimpliesirrelevant} \end{thm}

\begin{proof}
Suppose $v$ is fort-irrelevant. Then it is in no minimal fort of $G$, and therefore in no minimal cover of minimal forts of $G$.  By Proposition \ref{prop:minimalzfsisminimalcover}, then it is in no minimal zero forcing set of $G$, and is therefore zero-forcing irrelevant.  Conversely, if $v$ is zero-forcing-irrelevant, then it is in no minimal zero forcing set of $G$.  By Lemma  \ref{lem:inminimalfortinzfs}, then $v$ is in no minimal fort of $G$, and is therefore fort-irrelevant.
\end{proof}

Theorem \ref{thm:antifortimpliesirrelevant} allows us to apply the results from \cite{grood2024well} about zero-forcing-irrelevant vertices to fort-irrelevant vertices, establishing the following corollaries.  

\begin{cor}
In a tree $T$, a vertex is fort-irrelevant if and only if it is a star center. \label{cor:antifort} 
\end{cor}

\begin{cor}
In any graph $G$, every star center is fort-irrelevant.  \label{cor:starcenter}
\end{cor}

In addition to investigating which vertices are in no minimal fort, we are now able to determine which vertices appear in every minimal fort.  Note that a vertex is in every fort if and only if it is in every minimal fort.   
\begin{prop}
A vertex $v$ is in every (minimal) fort of $G$ if and only if it is an end vertex and $G$ is a path. 
\end{prop}
\begin{proof}
Let $v \in V(G)$ for some graph $G$.
Then $v$ is in every minimal fort of $G$ if and only if $\{v\}$ is a minimal cover of minimal forts of $G$.  By Proposition \ref{prop:minimalzfsisminimalcover}, the set $\{v\}$ is a minimal cover of minimal forts of $G$ if and only if $\{v\}$ is a minimal zero forcing set of $G$.  It is well known that the only graph with a minimal zero forcing set of order one is a path, and the set consists of one end vertex. 
\end{proof}

Note that a vertex is in every minimal fort if and only if it is in no maximal failed set, giving us the following characterization of failed-zero-forcing-irrelevant vertices, vertices that are in no maximal failed set.
\begin{cor}
A vertex $v$ is failed-zero-forcing-irrelevant if and only if $v$ is an end vertex and $G$ is a path.  \label{cor:fzfirrelevant}
\end{cor}

Since we established that a vertex is fort-irrelevant if and only if it is zero-forcing-irrelevant, and these are the main types of irrelevant vertices we use throughout the paper, hereafter we refer to fort-irrelevant and zero-forcing-irrelevant vertices as simply \emph{irrelevant} vertices.

\section{Leafy graphs and well-failed trees}
In this section, we characterize well-failed trees, and establish a large family of well-failed graphs.  
First, we introduce a particular kind of graph that we show is well-failed.  

\begin{dfn}
Let $G$ be a graph such that every vertex that is not a leaf has a double pendant.  Then we call $G$ a \emph{leafy graph}.  
\end{dfn}
An example of a leafy graph is shown in Figure \ref{fig:leafy}.

\begin{figure}[htbp]
\begin{center}
\begin{tikzpicture}[auto, scale=0.7]
\tikzstyle{vertex}=[draw, circle, inner sep=0.7mm]
\node (v1) at (-0.9, 0.7) [vertex]{};
\node(v1a) at  (-1.8, 1) [vertex]{};
\node(v1b) at  (-1.8, 0.4) [vertex]{};
\node (v2) at (-0.9, -0.7) [vertex]{};
\node(v2a) at  (-1.8, -0.3) [vertex]{};
\node(v2b) at  (-1.8, -0.7) [vertex]{};
\node(v2c) at  (-1.8, -1.1) [vertex]{};
\node (v3) at (0,0) [vertex]{};
\node(v3a) at  (-0.2, -0.8) [vertex]{};
\node(v3b) at  (0.2, -0.8) [vertex]{};
\node (v4) at (1,0) [vertex]{};
\node(v4a) at  (0.8, -0.8) [vertex]{};
\node(v4b) at  (1.2, -0.8) [vertex]{};
\foreach \x in {1, 2, 3, 4}
\draw (v\x) to (v\x a);
\foreach \x in {1, 2, 3, 4}
\draw (v\x) to (v\x b);
\draw (v1)--(v2)--(v3)--(v1);
\draw(v3)--(v4);
\draw(v2)--(v2c);
\end{tikzpicture}
\end{center}
\caption{A leafy graph.}
\label{fig:leafy}
\end{figure}
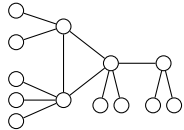

\begin{thm}
Every leafy graph is well-failed. \label{thm:leafy}
\end{thm}
\begin{proof}

If $G=K_2$, then the result is trivial.  Assume $|V(G)| \geq 3$.  By Theorem \ref{thm:antifortimpliesirrelevant}, every vertex in the graph that is not a leaf is an irrelevant vertex.  Thus, if $W$ is a minimal fort and $w \in W$, then $w$ is a leaf.  Let $v$ be the unique neighbor of $w$.  If $\deg(v)=1$, then $v \in W$ and $\{v, w\}$ is a minimal fort.  Otherwise, since $v \notin W$, $w \in W$,  and $W$ is a fort, the vertex $v$ must have at least one more neighbor $x \in W$, and like $w$, the vertex $x$ must be a leaf.  By Observation \ref{obs:doublependants}, $W = \{w,x\}$.  Since the choice of $w$ was arbitrary, every minimal fort in $G$ has two vertices.
\end{proof}

A \emph{generalized star} is a tree with exactly one high-degree vertex, which we refer to as the \emph{center} of the generalized star.  We refer to each pendent path from the vertex adjacent to the center to a leaf as a \emph{leg} of the generalized star.  

On a tree $T$ with at least one high-degree vertex, let $L_1$ and $L_2$ be two pendent paths from the same vertex $v$. Then the \emph{standard coloring of $T$ on $L_1$ and $L_2$} is the coloring of $T$ consisting of the standard coloring of the paths $L_1$ and $L_2$, starting with the vertices adjacent to $v$, and all other vertices in $T$ blue.  If $T$ is a generalized star with no legs of length one, then the \emph{adjusted coloring of $T$} is the coloring with the center of $T$ white, all neighbors of the center of $T$ blue, and then the standard labeling of a path on  each remaining subpath.



\begin{prop} Let $T$ be a tree with at least one high-degree vertex.  Let $L_1, L_2$ be pendent paths from the same high-degree vertex.  Then  the white vertices of the standard coloring of $T$ on $L_1$ and $L_2$ form a minimal fort on $T$.   \label{lem:standardgenstar}
\end{prop}

\begin{proof}
The only white vertices are those on legs $L_1$ and $L_2$.  Since the blue vertices on $L_1, L_2$ are in the standard labeling of a path, each is adjacent to exactly two white vertices.  The only other blue vertex in $T$ adjacent to a white vertex is the vertex $v$ from which $L_1, L_2$ are pendent paths, but the neighbors of $v$ on $L_1$ and $L_2$ are both white.   

 If we change any white vertex on, say, $L_1$ to blue, then all of $L_1$ will be forced, and $v$ will then force $L_2$.  Hence, the white vertices form a minimal fort.  \end{proof}

\begin{prop}
 If $T$ is a generalized star with no legs of length one,  then the white vertices of the adjusted coloring of $T$ form a minimal fort. \label{lem:adjusted}
\end{prop}

\begin{proof}
If $T$ has no legs of length one, the white vertices of the adjusted coloring form a fort: the center vertex of $T$ is white, but each of its neighbors has another white vertex as its other neighbor.  The remainder of each leg is labeled with the standard labeling of the path.  The blue vertices form a maximal failed set: if we change the center vertex to blue, then the blue vertex adjacent to the center on each leg will force the entire leg.  If we change any white vertex on a leg to blue, then the entire leg will be forced blue, and then the center vertex, which will force the whole graph to be blue.   By Lemma \ref{lem:equivalent}, the white vertices form a minimal fort.  
\end{proof}

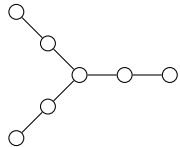
\begin{figure}[htbp]
\begin{center}
\begin{tikzpicture}[auto, scale=0.6]
\tikzstyle{vertex}=[draw, circle, inner sep=0.7mm]
\node(center) at (0,0) [vertex]{};
\node (v1a) at (1, 0)[vertex]{};
\node (v1b) at (2,0)[vertex]{};
\node (v2a) at (-0.7, 0.7)[vertex]{};
\node (v2b) at (-1.4, 1.4)[vertex]{};
\node (v3a) at (-0.7, -0.7)[vertex]{};
\node (v3b) at (-1.4, -1.4)[vertex]{};
\foreach \x in {1, 2, 3} {
\draw (center) to (v\x a);
\draw (v\x a) to (v\x b);
}
\end{tikzpicture}
\label{fig:specialt}
\caption{The only well-failed tree that has a high-degree vertex and is not leafy, $\specialt$}
\end{center}
\end{figure}

We refer to the tree shown in Figure \ref{fig:specialt} resulting from subdividing every edge of the star $K_{1,3}$ one time  as $\specialt$.  
\begin{lem}
The tree $\specialt$ is well-failed.\label{lem:specialgenstar}
\end{lem}

\begin{proof}
Let $v$ be the center of $T$.  The only maximal failed sets of $T$ are the blue vertices of the standard coloring of $T$ (on any pair of legs since all legs are of equal length) and the adjusted coloring.  Both resulting forts have four vertices.
\end{proof}

\begin{lem}
Suppose $T$ is a generalized star with three or more legs. Then $T$ is well-failed if and only if either $T=K_{1,n}$ or $T=\specialt$.  \label{lem:genstar_wellfailed}
\end{lem}

\begin{proof}
By Theorem \ref{thm:leafy} and Lemma \ref{lem:specialgenstar}, if $T=K_{1,n}$ or $T=\specialt$, then $T$ is well-failed.  
Assume that $T \notin \{ K_{1,n}, \specialt \}$.  Suppose $T$ has a leg of length one. Since $T$ is not a star, it must have at least one longer leg as well.  Let $L_1, L_2$ be  two shortest legs, and let $L_3, L_4$ be two longest, noting that $L_3=L_2$ if $T$ has only three legs.  Let $W$ be the white vertices of the standard coloring on legs $L_1, L_2$, and $W'$ be the white vertices of the standard coloring on legs $L_3, L_4$.  Then $W$ and $W'$ are  minimal forts by Proposition \ref{lem:standardgenstar}, and $|W|<|W'|$, so $T$ is not well-failed.

We can assume for the remainder of the proof that all legs have length at least two.    Let $W$ again be the white vertices of the standard coloring on two shortest legs $L_1, L_2$.  Let $\hat{S}$ be the white vertices of the adjusted coloring on $T$.  Then $W,  \hat{W}$ are minimal forts by Propositions \ref{lem:standardgenstar} and \ref{lem:adjusted}.  We now show that $|W|<|\hat{W}|$ to complete the proof. Let $T'$ be the subgraph of $T$ induced by the center of $T$ and vertices of $L_1, L_2$.  Then $|V(T') \cap W|-1 \leq |V(T') \cap \hat{W}| \leq |V(T') \cap W|+1$.    However, $T$ has at least one leg $L'$ distinct from $L_1$ and $L_2$ that has length $k \geq 2$.  Note that if $L'$ is the only such leg, that the length of $L'$ is at least three, since otherwise $T=\specialt$.  We can see that $|W|<|\hat{W}|$ since $|\hat{W} \backslash V(T')| \geq 2$ and $|\hat{W} \backslash V(T')|=0$.  Thus, $T$ is not well-failed. 
\end{proof}

A \emph{pendent generalized star} in a graph $G$ is an induced subgraph $R$ of $G$ such that exactly one vertex $v$ of $R$ is high degree in $G$; for some $k \geq 2$, $G-v$ has $k+1$ components, exactly $k$ of which are pendent paths of $v$; and $R$ is induced by $v$ and the vertices of the $k$ pendent paths, which we refer to as \emph{legs}.  We call $v$ the \emph{center} of the pendent generalized star.  

It is well-known that every tree with at least two high-degree vertices has a pendent generalized star, and a formal proof was provided in \cite{fallat2007minimum}.  The same proof can be modified to show the following.  We include the modification to the proof for completeness.  
\begin{lem}
Every tree with at least two high-degree vertices has at least two pendent generalized stars.  \label{lem:twopgs}
\end{lem}

\begin{proof}
Let $T'$ be the subtree of $T$ created by deleting all pendent paths of $T$.  Since $T$ has at least two high-degree vertices, $|V(T')|\geq 2$.  Since $T'$ is a tree, it has at least two leaves, $u$ and $v$.   Then in $T$, $u$ and $v$ each form pendent generalized stars together with their pendent paths. 
\end{proof}

An \emph{adjusted pgs coloring} of a pendent generalized star $R$ consists of coloring the center of $R$ white, and then repeating a blue-white alternating pattern on each leg starting with the vertex adjacent to the center and proceeding toward each leaf, with every leaf colored white.   A tree  $T$ is a \emph{double generalized star} if $T$ has precisely two high-degree vertices, and they are adjacent.

\begin{lem}
 If $T$ is a double generalized star with no legs of length one, then the white vertices of the adjusted pgs colorings of both pendent generalized stars together form a minimal fort. \label{lem:adjustedpgsdouble}
\end{lem}

\begin{proof}
We see that every blue vertex has two white neighbors in $T$.  Suppose we change some white vertex to blue.  If it is a high-degree vertex, then every leg of its pendent generalized star will be forced, and the high-degree vertex will force its high-degree neighbor, which will then force all of its legs.  If we change a white vertex on any pendent path to blue, we now have either a blue leaf or two blue adjacent vertices on a leg, forcing the leg to blue along with the center of its pendent generalized star, and as above, the blue vertices form a zero forcing set.    Hence, the white vertices of the adjusted pgs coloring of both pendent generalized stars form a minimal fort.  
 \end{proof}

\begin{lem}
Suppose $T$ is a double generalized star.  Then $T$ is well-failed if and only if $T$ is leafy. \label{lem:doublestar}
\end{lem}

\begin{proof}
We know that the reverse direction holds by  Theorem \ref{thm:leafy}.  For the forward direction, suppose that $T$ is not leafy.  Then if each high-degree vertex has a double pendant, at least one must have a longer leg, $L$.  We can construct a minimal fort $W$ using the standard coloring on $L$ and any other leg from the same high-degree vertex; since $L$ has length at least two, $|W| \geq 3$.  Since $T$ has a double pendant, there is also a minimal fort with two vertices in $T$.  Hence $T$ is not well-failed.

For the rest of the proof, we assume that each high-degree vertex has at most one leg of length one.  First assume that there are no legs of length one.  Let $W$ be the white vertices of the standard coloring on any two legs $L_1, L_2$ of the same pendent generalized star, and $W'$ the white vertices of the adjusted pgs colorings on both pendent generalized stars of $T$.  The sets $W$ and $W'$ are minimal forts by Lemmas \ref{lem:standardgenstar} and \ref{lem:adjustedpgsdouble} respectively. Note that $\left|W \cap V(L_1)\right| -1 \leq   \left|W' \cap V(L_1) \right| \leq \left|W \cap V(L_1) \right|$ and similarly on $L_2$.  However, $\left|W' \backslash \left(V(L_1) \cup V(L_2)\right)\right| \geq4$ while $\left|W \backslash \left(V(L_1) \cup V(L_2)\right)\right| =0$, so  $|W'| > |W|$. 

Next, assume that $T$ has a leg of length one, but no more than one length-one leg for each high-degree vertex.  Let $v$ be a high-degree vertex that has a length one leg, $w$ its leaf neighbor, and $R_v$ the pendent generalized star of which $v$ is the center. Let $u$ be the other high-degree vertex (that may or may not have a length-one leg) and $R_u$ its pendent generalized star.  Let $X$ be the white vertices of the adjusted pgs coloring of $R_u$.  Let $W_1 = X \cup \{w\}$.  Take a longest leg of $R_v$, and color the vertices with the standard coloring of the path starting with the vertex adjacent to $v$.  Let $Y$ be the white vertices of this leg and let $W_2 = X \cup Y$.  We can see that $W_1$ is a fort: the only vertex in $V(R_v) \cap W_1$ is $w$.  Its neighbor is $v \notin W_1$, and $v$ has two neighbors, $w, u$ in $W_1$.  Every  vertex in $V(R_u) \backslash W_1$ has exactly two neighbors in $V(R_u) \cap W_1$, so $W_1$ is a fort.  For $W_2$, note that each vertex in $V(R_v) \backslash W_2$ has two neighbors in $W_2$, and the same argument holds for $V(R_u)\cap W_2$ as with $W_1$.  For minimality, note in both cases that if we remove any vertex from $W_i \cap V(R_u)$ that $u$ will either be blue or be forced by the leg with the change, and then all of $R_u$ will turn blue; then $v$ forces its leg with vertices from $W_i$ to turn blue. If a vertex from $R_v$ is removed from $W_i$, then all of $R_v$ will turn blue, $u$ will become blue, and then all of $R_u$.  Hence, $W_1, W_2$ are minimal forts, and $|W_2|>|W_1|$.    Hence, $T$ is not well-failed.  
\end{proof}

\begin{lem}
Suppose $W$ is a fort in a tree $T$ where $|V(T)| \geq 2$.  If $w \in W$, then there exists a path $P$ from a leaf $x_1$ to another leaf $x_2$ of $T$ containing $w$ such that $x_1, x_2 \in W$, and any vertex $v \in V(P) \backslash W$ has a neighbor on either side in $P$ that is in $W$.  \label{lem:leaftoleafcontainingx} 
\end{lem}

\begin{proof}
Let $w \in W$. Construct $P$ as follows.  Consider any neighbor of $w$.  If the neighbor is in $W$ add it to $P$.  Take the next such neighbor and the next, adding each to $P$ until we reach either a leaf, or a neighbor $v_1 \notin W$, which we add to $P$.  Now, $v_1$ has a neighbor in $W$ (along the path we've constructed so far), so it must have a second neighbor in $W$.  Call that neighbor $w_1$ and add it to $P$.  We continue this process for $w_1$ that we did for $w$ until we finally reach a leaf.  

If $w$ is a leaf, then we're done. Otherwise, we can perform this process starting with another neighbor of $w$, and complete construction of $P$.  
\end{proof}

\begin{cor}
If $w \in W$ for a minimal fort $W$ in a tree $T$, and $w$ is not a vertex in a double pendant, then $|W| \geq 3$.  \label{cor:biggerthantwo}
\end{cor}

The following theorem is the direct result of combining Theorem \ref{thm:antifortimpliesirrelevant} and \cite[Theorem 2.10]{grood2024well}.

\begin{thm}
Suppose $T$ has no double pendants. Then every vertex of $v$ is in some minimal fort.  \label{thm:nodoublependantstree}
\end{thm}

We will also use the following lemmas to characterize well-failed trees.

\begin{lem}
Let $T$ be a tree with no double pendants.  Then for any pair of leaves $v_1, v_2$ in $T$, we can construct a minimal fort containing $v_1$ and $v_2$.  \label{lem:anyleaftoleaf}
\end{lem}                                                                                                                                                                                                                                                                                                                                                                                                                                                                                                                                                                                                                                                                                                                                                                                                                                                                                                                                                                                                                                                                                                                                                                                                                                                                                                                                                                                                                                                                              

\begin{proof}
We know that the statement holds for paths and generalized stars.  We use induction on the number of high-degree vertices in $T$.  Suppose the result holds if $T$ has up to $k \geq 1$ high-degree vertices.  Let $T$ be a tree with $k+1$ high-degree vertices.  Then $T$ has a pendent generalized star, $R$.  Let $v$ be the center of $R$.
Let $T'$ be the tree formed by removing all but the shortest leg from $R$, and let $x$ be the end vertex of the remaining leg of $R$ in $T'$.  Then in $T'$, the vertex $v$ has degree two, giving us that $T'$ has $k$ high-degree vertices, and we can apply the induction hypothesis to $T'$.  Consider any pair of leaves in $T$.  If both leaves are in $R$, then we can use the standard coloring on the two legs containing the leaves to produce a minimal fort.  If both leaves are outside of $R$, let $W'$ be the minimal fort on $T'$ that contains them.  If $v \notin W'$, then $W'$ is also a minimal fort on $T$ and we're done.  Otherwise, if $v \in W'$, then let $X$ be the white vertices of the adjusted pgs coloring on $R$, and let $W=W' \cup X$.  Since $T$ has no double pendants, all legs of $R$ in $T \backslash T'$ have length greater than one, and $W$ is a minimal fort on $T$.   

Finally, suppose $v_1 \in R$, and $v_2 \notin R$.  We know that there is a minimal fort $W'$ of $T'$ containing the leaves $x$ and $v_2$.  If $v \notin W'$ and $x=v_1$, then $W'$ is a minimal fort in $T$ as well.  If $v \notin W'$ and $x \neq v_1$, then let $Y'$ be the white vertices of $W'$ on the leg containing $x$, and let $Y$ be the  white vertices of the standard coloring of the path on the leg containing $v_1$.   Then $W= (W' \cup Y) \backslash Y'$ is a minimal fort on $T$.  If $v \in W'$, then define $W= W' \cup X$ where $X$ is defined as in the case above, and we're done.
 \end{proof}

The following is a result of combining \cite[Corollary 2.13]{grood2024well} with Theorem \ref{thm:antifortimpliesirrelevant}.

\begin{lem}
If $T$ is a tree and $T'$ the remaining graph after a star removal is performed on $T$, then a vertex $u \in V(T')$ is irrelevant  in $T'$ if and only if it is irrelevant in $T$.  \label{lem:starremovalirrelevant}
\end{lem}

We are now ready for the main result of this section: a characterization of well-failed trees.

\begin{thm}
A tree $T$ is well-failed if and only if one of the following holds: $T = P_n$ where $n \leq 4$ or $n=6$,   $T = \specialt$, or $T$ is leafy.  
 \label{thm:tree}
\end{thm}

\begin{proof}

The reverse direction is established by  Lemma \ref{lem:path}, Theorem \ref{thm:leafy}, and Lemma  \ref{lem:specialgenstar}.

For the forward direction,  note that if $T=P_n$ with $n \notin \{1, 2, 3, 4,6\}$ that $T$ is not well-failed by Lemma \ref{lem:path}.  For the rest of the proof, we assume that $T \neq \specialt$ is a tree with at least one high-degree vertex and that $T$ is not leafy.

Suppose that $T$ has some double pendants.  Let $T'$ be the tree remaining after doing star removals on $T$ for all vertices in $B_0$.  Note that $V(T')$ is nonempty because $T$ is not leafy.  If $B_1$ is empty, then by Theorem \ref{thm:nodoublependantstree} and Lemma \ref{lem:starremovalirrelevant},  every vertex in $V(T')$ is in a minimal fort of $T$.   Let $x \in V(T')$, and note that $x$ is not part of a double pendant in $T$ because $x \in V(T')$, so by Corollary \ref{cor:biggerthantwo}, any minimal fort of $T$ containing $x$ has at least three vertices, and $T$ is not well-failed.   If $B_1$ is not empty, let $v$ be a star center in $B_1$, and let $w$ be a neighbor of $v$ that is a leaf in $T'$ but not in $T$.  Then $w$ is in a minimal fort of $T'$ of order two since it's a vertex in a double pendent in $T'$, and by Lemma \ref{lem:starremovalirrelevant} and Corollary \ref{cor:biggerthantwo}, the vertex $w$ is in a minimal fort of $T$ with more than two vertices. Hence $T$ is not well-failed.

Now suppose $T$ has no double pendants.  That is, no star removal is possible.   Then every vertex in $V(T)$ is in some minimal fort by Theorem \ref{thm:nodoublependantstree}.  We know from Lemmas \ref{lem:genstar_wellfailed} and  \ref{lem:doublestar} that if $T$ is a generalized star or a double generalized star that the result holds.  Otherwise, $T$ has at least two pendent generalized stars, $R_1$ and $R_2$ by Lemma \ref{lem:twopgs}.  Let $v_1, v_2$ be the centers of $R_1, R_2$ respectively.  We can take standard colorings on all pairs of legs of $R_1$ and  for $R_2$ to create a collection of minimal forts.  If any of these forts have different cardinalities, then we're done.  Otherwise, let $r$ be the number of vertices of any standard minimal fort on $R_1$ or $R_2$.  Note that $r>2$ because $T$ has no double pendants.  

Let $w_1$ be a leaf on a longest leg of $R_1$ and $w_2$ a leaf on a longest leg of $R_2$.  By Lemma \ref{lem:anyleaftoleaf}, there exists a minimal fort $W$ containing $w_1$ and $w_2$.  If $v_1 \notin W$, then note that $|V(R_1) \cap W| \geq r/2$.  If $v_1 \in W$, then $|V(R_1) \cap W| \geq r-1$.  Similar statements hold for $v_2, R_2$ as well.  

Since $R_1, R_2$ are both pendent generalized stars and $T$ is not a double generalized star, the centers $v_1, v_2$ are not adjacent.  If neither of $v_1, v_2$ is in $W$, then there must be at least one more vertex from $W$ between $v_1$ and $v_2$, giving us that $|W| \geq r+1$.  If $v_1 \in W$ and $v_2 \notin W$, we also must have at least one more vertex from $W \backslash \left( V(R_1) \cup V(R_2) \right)$, giving us that $|W| \geq r-1 + r/2 +1 > r$.  Finally, if $v_1, v_2 \in W$, then $|W| \geq 2(r-1) > r$.  Hence in all cases, there is a minimal fort larger than $r$, giving us that $T$ is not well-failed. 
\end{proof}

A tree is well-forced if and only if performing star removals until no more star removals can be performed results in a possibly empty set of copies of $K_2$  \cite[Theorem 3.9]{grood2024well}.  Since performing star removals on a leafy tree results in an empty graph,  together with Theorem \ref{thm:tree} we have the following corollary. 

\begin{cor}
Any tree $T$ that is well-failed and is not a path or $\specialt$ is well-forced. 
\end{cor}

\section{Other well-failed graphs}
We have established that leafy graphs are well-failed, and for trees, found a characterization.  Here we first determine which cycles are well-failed, and then identify other well-failed graphs that contain cycles and are not leafy.  

\begin{prop}
A cycle $C_n$ is well-failed if and only if $n \leq 5$ or $n=7$.
\end{prop}
\begin{proof}
For $C_3$, any one vertex is a failed set.  For $C_4, C_5$, take any pair of nonadjacent vertices to form a maximum failed set, and any single vertex is not a maximal failed set.  For $n=7$, any set of three nonadjacent vertices forms a maximum failed set.  If $u, v$ are a pair of nonadjacent vertices in $C_7$, there is a path with three internal vertices from $u$ to $v$, meaning we could add the middle vertex of that path to $u, v$ to produce a failed set of order three.  Hence $C_7$ is well-failed.

Otherwise, for $n \geq 6$ with $n \neq 7$, label the vertices $v_1, v_2, \ldots, v_n$ around the cycle.   Let $S=\{v_i : i \mbox{ is odd and } i <n\}$, and let $S'=S \cup \{v_4\}\backslash \{v_3, v_5\}$.  Then $|S'|>|S|$, and note that both sets are maximal failed sets because they contain no pair of adjacent vertices, but all vertices outside of $S, S'$ are adjacent to a vertex in $S, S'$.  Hence $C_n$ is not well-failed for $n  \geq 5$ with $n \neq 7$.  
\end{proof}

\begin{obs}
Let $G$ be a graph and $W$ a fort in $G$. If $N[v] \subseteq W$ for some $v \in V(G)$ with $|N(v)|\geq 2$, then $W$ is not a minimal fort.  \label{obs:closedneighborhoodnotfort}
\end{obs}

\begin{figure}[htbp]
\begin{center}
\begin{tikzpicture}[auto,scale=0.7] \tiny
\tikzstyle{vertex}=[draw, circle, inner sep=0.7mm]
\node (v1) at (90: 0.7) [vertex,  label=right:$u_1$] {};
\node (v2) at (90+1*72:0.7) [vertex, label=below:$u_5$] {};
\node (v3) at (90+2*72: 0.7) [vertex,   label=  below:$u_4$] {};
\node (v4) at (90+3*72: 0.7) [vertex, label= below:$u_3$] {};
\node (v5) at (90+4*72: 0.7)  [vertex, label=below: $u_2$] {};
\node (v6) at (90: 1.5) [vertex,   label=$x_1$] {};
\node (v7) at (90+1*72:1.5) [vertex, label=$x_5$] {};
\node (v8) at (90+2*72: 1.5) [vertex,  label=  below:$x_4$] {};
\node (v9) at (90+3*72: 1.5) [vertex, label= below:$x_3$] {};
\node (v10) at (90+4*72:1.5) [vertex, label= $x_2$] {};
\foreach[evaluate=\y using int(\x-1)] \x in {7, 8, 9, 10}
\draw (v\y) to (v\x);
\foreach[evaluate=\y using int(\x-2)] \x in {3, 4, 5}
\draw (v\y) to (v\x);
\foreach[evaluate=\y using int(\x-3)] \x in {4, 5}
\draw (v\y) to (v\x);
\foreach[evaluate=\y using int(\x+5)] \x in {1, 2, 3, 4, 5}
\draw (v\y) to (v\x);
\draw (v6) to (v10);
\end{tikzpicture}
\caption{The labeled Petersen graph}
\label{fig:petersen}
\end{center}
\end{figure}
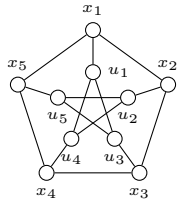

\begin{prop}
The following graphs are well-failed.
\begin{enumerate}
\item Any complete graph.
\item Any complete bipartite graph.  
\item The Petersen graph.
\end{enumerate}
\end{prop}

\begin{proof}
\begin{enumerate}
\item For $K_1$, the result is trivial.  For $n\geq 2$, note that any pair of vertices forms a minimum fort in $K_n$.  
\item Let $G=K_{m,n}$ with $m \leq n$.  We know the result holds if $m =1$ by Theorem \ref{thm:tree}, so assume $m \geq 2$.  If $\{v_1, v_2\}$ is a pair of vertices in the same partite set, then $N(v_1)=N(v_2)$.  Hence $\{v_1, v_2\}$ forms a fort that is minimal because $|\{v_1, v_2 \}|=2$.  Any set containing $\{v_1, v_2\}$ as a proper subset is not a minimal fort.  Thus all minimal forts have order $2$, and $K_{m,n}$ is well-failed.    

\item Let $G$ be the Petersen graph.  We label the vertices as shown in Figure \ref{fig:petersen}: draw the graph as an outer $C_5$ drawn in its typical cycle form and an inner $C_5$ drawn as a five-point star.  Let $X$ be the vertices of the outer cycle, labeled in order around the cycle $x_1$ through $x_5$.  Let $U$ be the vertices of the inner cycle, with $u_i$ the vertex in $U$ that is adjacent to $x_i \in X$.  Note that choice of ``outer'' versus ``inner'' cycle is arbitrary.

It is known that $\Z(G)=5$ and $\F(G)=6$ \cite{aim2008zero, fetcie2015failed}.  Thus, any minimum fort has four vertices.  We will show that all minimal forts have four vertices.  Assume that some fort $W$ has more than four vertices. Then without loss of generality, $|X \cap W| \geq 3$.  

If $|X \cap W| = 5$, then by Observation \ref{obs:closedneighborhoodnotfort}, we must have $U \cap W = \emptyset$, but then $W$ is not a fort since each vertex in $U$ has one neighbor in $W$.  Hence $|X \cap W| <5$.

Suppose $|X \cap W| = 4$.  Let $x_5$ be the lone vertex in $X \backslash W$.  By Observation \ref{obs:closedneighborhoodnotfort}, $u_2, u_3 \notin W$.  Now, $\{u_1, u_2, u_3, u_4, x_5\}$ is a zero forcing set, so if $W$ is a fort, we must have at least one of $\{u_1, u_4\}$ in $W$.  However, $\{x_1, x_3, x_4, u_1\}$ and $\{x_1, x_2, x_4, u_4\}$ are both forts, and $W$ cannot be a superset of either since we assumed it is a minimal fort, giving us that $u_1, u_4 \notin W$.    Hence, if $|X \cap W| = 4$, then $W$ is not a minimal fort.

Suppose $|X \cap W| =3$.  We can assume without loss of generality that either $X \cap W = \{x_1, x_2, x_3\}$ or $X \cap W = \{x_1, x_2, x_4\}$.  We start with $X \cap W = \{x_1, x_2, x_3\}$.  Then, by Observation  \ref{obs:closedneighborhoodnotfort}, $u_2 \notin W$.  Suppose $u_1, u_3 \notin W$.  If $u_4 \notin W$, then $V(G) \backslash W$ is a zero forcing set and similarly for $u_5$, so  $W = \{x_1, x_2, x_3, u_4, u_5\}$.  However, $\{x_1, x_3, u_4, u_5\}$ is a fort, so if $W$ is minimal, at least one of $u_1, u_3 \in W$.  Assume without loss of generality $u_1 \in W$.  If $u_4 \in W$, then $\{x_2, x_3, u_1, u_4\}$ is a subset of $W$ and a fort, so we must have $u_4 \notin W$, but then $V(G) \backslash W$ is a zero forcing set.   

Finally, suppose $X \cap W = \{x_1, x_2, x_4\}$.  Since we assume $W$ is a minimal fort, it cannot be a superset of any of the following forts: $W_1=\{x_1, x_2, u_3, u_5\}, W_2=\{x_1, x_2, u_4, x_4\}, W_3=\{x_2, x_4, u_1, u_5\},  W_4=\{x_1, x_4, u_2, u_3\}$.  Because of $W_2$, then $u_4 \notin W$.  And $W_1, W_3, W_4$ give us that one of the following pairs is not in  $W$: $S_1=\{u_1, u_3\}, S_2=\{u_3, u_5\}, S_3=\{u_2, u_5\}$. Note that because of symmetry, we need only consider $S_1, S_2$.  If $S_1\cap W = \emptyset$ or $S_2\cap W = \emptyset$, then we have $\{x_3, x_5, u_1, u_3, u_4\} \subseteq V(G) \backslash W$ or $\{x_3, x_5, u_3, u_4, u_5\} \subseteq V(G) \backslash W$ respectively, but each is a zero forcing set, so $W$ is not a minimal fort. 

Hence, all minimal forts of $G$ have four vertices.
\end{enumerate}
\end{proof}

\bibliography{wellfailedbib}

\begin{thebibliography}{10}

\bibitem{aim2008zero}
{AIM Minimum Rank-Special Graphs Work Group}.
\newblock Zero forcing sets and the minimum rank of graphs.
\newblock {\em Linear Algebra Appl.}, 428(7):1628--1648, 2008.

\bibitem{anderson2021well}
Sarah~E. Anderson, Kirsti Kuenzel, and Douglas~F. Rall.
\newblock On well-dominated graphs.
\newblock {\em Graphs Combin.}, 37(1):151--165, 2021.

\bibitem{beaudouin2020zero}
Matthew Beaudouin-Lafon, Margaret Crawford, Serena Chen, Nathaniel Karst,
  Louise Nielsen, and Denise~Sakai Troxell.
\newblock On the zero blocking number of rectangular, cylindrical, and
  {M}{\"o}bius grids.
\newblock {\em Discrete Applied Mathematics}, 282:35--47, 2020.

\bibitem{bong2023isomorphisms}
Novi~H. Bong, Joshua Carlson, Bryan Curtis, Ruth Haas, and Leslie Hogben.
\newblock Isomorphisms and properties of {TAR} graphs for zero forcing and
  other {$X$}-set parameters.
\newblock {\em Graphs Combin.}, 39(4):Paper No. 86, 23, 2023.

\bibitem{brimkov2022minimal}
Boris Brimkov and Joshua Carlson.
\newblock Minimal zero forcing sets.
\newblock {\em Australasian Journal of Combinatorics}, 90(3):363--377, 2024.

\bibitem{brimkov2019computational}
Boris Brimkov, Caleb~C. Fast, and Illya~V. Hicks.
\newblock Computational approaches for zero forcing and related problems.
\newblock {\em European Journal of Operational Research}, 273(3):889--903,
  2019.

\bibitem{burgarth2009local}
Daniel Burgarth, Sougato Bose, Christoph Bruder, and Vittorio Giovannetti.
\newblock Local controllability of quantum networks.
\newblock {\em Physical Review A—Atomic, Molecular, and Optical Physics},
  79(6):060305, 2009.

\bibitem{burgarth2013zero}
Daniel Burgarth, Domenico D'Alessandro, Leslie Hogben, Simone Severini, and
  Michael Young.
\newblock Zero forcing, linear and quantum controllability for systems evolving
  on networks.
\newblock {\em IEEE Transactions on Automatic Control}, 58(9):2349--2354, 2013.

\bibitem{CZ}
Gary Chartrand and Ping Zhang.
\newblock {\em A First Course in Graph Theory}.
\newblock Cambridge Tracts in Mathematics. Courier Corporation, 2012.

\bibitem{fallat2007minimum}
Shaun~M Fallat and Leslie Hogben.
\newblock The minimum rank of symmetric matrices described by a graph: A
  survey.
\newblock {\em Linear Algebra and its Applications}, 426(2-3):558--582, 2007.

\bibitem{fast2017novel}
Caleb~C. Fast.
\newblock {\em Novel Techniques for the Zero-Forcing and $p$-Median Graph
  Location Problems}.
\newblock PhD thesis, Rice University, 2017.

\bibitem{fast2018effects}
Caleb~C. Fast and Illya~V. Hicks.
\newblock Effects of vertex degrees on the zero-forcing number and propagation
  time of a graph.
\newblock {\em Discrete Applied Mathematics}, 250:215--226, 2018.

\bibitem{fetcie2015failed}
Katherine Fetcie, Bonnie Jacob, and Daniel Saavedra.
\newblock The failed zero forcing number of a graph.
\newblock {\em Involve}, 8(1):99--117, 2015.

\bibitem{finbow1988well}
Arthur~S. Finbow, Bert~L. Hartnell, and Richard Nowakowski.
\newblock Well-dominated graphs: a collection of well-covered ones.
\newblock volume~25, pages 5--10. 1988.
\newblock Eleventh British Combinatorial Conference (London, 1987).

\bibitem{grood2024well}
Cheryl Grood, Ruth Haas, Bonnie~C. Jacob, Erika~L.C. King, and Shahla
  Nasserasr.
\newblock Well-forced graphs.
\newblock {\em Journal of Graphs and Combinatorics}, 40(129), 2024.

\bibitem{hartnell1999well}
Bert~L. Hartnell.
\newblock Well-covered graphs.
\newblock {\em J. Combin. Math. Combin. Comput.}, 29:107--115, 1999.

\bibitem{hogben2022inverse}
Leslie Hogben, Jephian C.-H. Lin, and Bryan~L. Shader.
\newblock {\em Inverse problems and zero forcing for graphs}, volume 270 of
  {\em Mathematical Surveys and Monographs}.
\newblock American Mathematical Society, Providence, RI, 2022.

\bibitem{plummer1993well}
Michael~D. Plummer.
\newblock Well-covered graphs: a survey.
\newblock {\em Quaestiones Math.}, 16(3):253--287, 1993.

\end{thebibliography}
\bibliographystyle{plain}
\end{document}